\documentclass{article}

\usepackage{everyday_math}
\usepackage{wasysym}
\usepackage{cleveref}
\usepackage[T1]{fontenc}
\usepackage{times}
\usepackage{mathrsfs}
\usepackage{ltxtable}
\usepackage{tikz-cd}
\usepackage[super]{nth}

\usepackage{geometry}
\ifpdf
    \pdfcatalog { /PageMode (/UseOutlines)
                  /OpenAction (fitbh)  }
\fi

\title{Kalimullin Pair and Semicomputability in $\alpha$-Computability Theory}
\ifpdf
  \author{\href{mailto:mmdt@leeds.ac.uk}{D\'avid Natingga}}

\else
  \author{D\'avid Natingga}
\fi

\newcommand{\cki}{\omega^{CK}_1}

\newcommand{\adeg}[1]{\mathrm{deg}_{\alpha}({#1})}
\newcommand{\aedeg}[1]{\mathrm{deg}_{\alpha e}({#1})}

\newcommand{\aeop}[1]{\Phi_{#1}}

\newcommand{\asc}{\mathrm{sc}(L_\alpha)} %alpha-semicomputable

\newcommand{\ukpair}[3]{\mathcal{K}_{#1}(#2, #3)}

\newcommand{\kpair}[2]{\mathcal{K}(#1, #2)}
\newcommand{\kpairnt}[2]{\mathcal{K}_{\mathrm{nt}}(#1, #2)}
\newcommand{\kpairmax}[2]{\mathcal{K}_{\mathrm{max}}(#1, #2)}

\newcommand{\hreals}{{}^*\mathbb{R}}

\newcommand{\mrtotalaedegs}{\mathcal{TOT}_{\alpha e}^{\mathrm{mr}}}

\begin{document}

\maketitle

%set the number of sectioning levels that get number and appear in the contents
\setcounter{secnumdepth}{2}
\setcounter{tocdepth}{2}

%\tableofcontents

%\mainmatter

\subsubsection{\Large Abstract}
We generalize some results on semicomputability by Jockusch \cite{jockusch1968semirecursive} to the setting of $\alpha$-Computability Theory.
We define an $\alpha$-Kalimullin pair and show that it is definable in the $\alpha$-enumeration degrees $\mathcal{D}_{\alpha e}$ if the projectum of $\alpha$ is $\alpha^*=\omega$ or if $\alpha$ is an infinite regular cardinal.
Finally using this work on $\alpha$-semicomputability and $\alpha$-Kalimullin pairs we conclude that every nontrivial total $\alpha$-enumeration degree is a join of a maximal $\alpha$-Kalimullin pair if $\alpha$ is an infinite regular cardinal.

\section{$\alpha$-Computability Theory}
$\alpha$-Computability Theory is the study of the definability theory over G\"odel's $L_\alpha$ where $\alpha$ is an admissible ordinal. One can think of equivalent definitions on Turing machines with a transfinite tape and time \cite{koepke2005turing} \cite{koepke2007alpha} \cite{koepke2009ordinal} \cite{koepke_seyfferth2009ordinal} or on generalized register machines \cite{koepke2008register}. Recommended references for this section are \cite{sacks1990higher}, \cite{chong1984techniques}, \cite{maass1978contributions} and \cite{di1983basic}.

Classical Computability Theory is $\alpha$-Computability Theory where $\alpha = \omega$.

\subsection{G\"odel's Constructible Universe}

\begin{defn}\label{defn_goedels_constructible_universe}(G\"odel's Constructible Universe)\\
Define \emph{G\"odel's constructible universe} as $L := \bigcup_{\beta \in \mathrm{Ord}} L_\beta$ where $\gamma, \delta \in \mathrm{Ord}$, $\delta$ is a limit ordinal and:

$L_0 := \emptyset$,

$L_{\gamma + 1} := \mathrm{Def}(L_\gamma):=\{x | x \subseteq L_\gamma $ and $x$ is first-order definable over $L_\gamma\}$,

$L_\delta = \bigcup_{\gamma < \delta} L_\gamma$.
\end{defn}

\subsection{Admissibility}

\begin{defn}(Admissible ordinal\cite{chong1984techniques})\\
An ordinal $\alpha$ is \emph{$\Sigma_1$ admissible} (admissible for short) iff $\alpha$ is a limit ordinal and $L_\alpha$ satisfies $\Sigma_1$-collection:
$\forall \phi(x,y) \in \Sigma_1(L_\alpha). L_\alpha \models \forall u [\forall x \in u \exists y. \phi(x,y) \implies \exists z \forall x \in u \exists y \in z. \phi(x,y)]$ where $L_\alpha$ is the $\alpha$-th level of the G\"odel's Constructible Hierarchy (\cref{defn_goedels_constructible_universe}).
\end{defn}

\begin{eg}(Examples of admissible ordinals \cite{chong1984techniques} \cite{takeuti1965recursive})
\begin{itemize}
\item $\cki$ - Church-Kleene $\omega_1$, the first non-computable ordinal
\item every stable ordinal $\alpha$ (i.e. $L_\alpha \prec_{\Sigma_1} L$), e.g. $\delta^1_2$ - the least ordinal which is not an order type of a $\Delta^1_2$ subset of $\mathbb{N}$, \nth{1} stable ordinal
\item every infinite cardinal in a transitive model of $\mathrm{ZF}$
\end{itemize}
\end{eg}

\subsection{Basic concepts}

\begin{defn}\label{defn_ez_fin}
A set $K \subseteq \alpha$ is \emph{$\alpha$-finite} iff $K \in L_\alpha$.
\end{defn}

\begin{defn}($\alpha$-computability and computable enumerability)
\begin{itemize}
\item A function $f:\alpha \to \alpha$ is \emph{$\alpha$-computable} iff $f$ is $\Sigma_1(L_\alpha)$ definable.
\item A set $A \subseteq \alpha$ is \emph{$\alpha$-computably enumerable} ($\alpha$-c.e.) iff $A \in \Sigma_1(L_\alpha)$.
\item A set $A \subseteq \alpha$ is \emph{$\alpha$-computable} iff $A \in \Delta_1(L_\alpha)$ iff $A \in \Sigma_1(L_\alpha)$ and $\alpha-A \in \Sigma_1(L_\alpha)$.
\end{itemize}
\end{defn}

\begin{prop}\label{prop_bij_alpha_to_l_alpha}\cite{chong1984techniques}
There exists a $\Sigma_1(L_\alpha)$-definable bijection $b:\alpha \to L_\alpha$.
\qed
\end{prop}

Let $K_\gamma$ denote an $\alpha$-finite set $b(\gamma)$. The next proposition establishes that we can also index pairs and other finite vectors from $\alpha^n$ by an index in $\alpha$.

\begin{prop}\label{prop_bij_to_n_fold_product}\cite{maass1978contributions}
For every $n$, there is a $\Sigma_1$-definable bijection $p_n$:$\alpha \to \alpha \times \alpha \times ... \times \alpha$ (n-fold product).
\qed
\end{prop}

Similarly, we can index $\alpha$-c.e., $\alpha$-computable sets by an index in $\alpha$.
Let $W_e$ denote an $\alpha$-c.e. set with an index $e < \alpha$.

\begin{prop}\label{prop_alpha_finite_union}($\alpha$-finite union of $\alpha$-finite sets\footnote{From \cite{sacks1990higher} p162.})\\
$\alpha$-finite union of $\alpha$-finite sets is $\alpha$-finite, i.e. if $K \in L_\gamma$, then $\bigcup_{\gamma \in K} K_\gamma \in L_\alpha$.
\qed
\end{prop}

\subsection{Enumeration reducibility}
The generalization of the enumeration reducibility corresponds to two different notions - weak $\alpha$-enumeration reducibility and $\alpha$-enumeration reducibility.

\begin{defn}(Weak $\alpha$-enumeration reducibility)\\
$A$ is weakly $\alpha$-enumeration reducible to $B$ denoted as $A \le_{w\alpha e} B$ iff $\exists \Phi \in \Sigma_1(L_\alpha)$ st
$\Phi(B) = \{x < \alpha : \exists \delta < \alpha [ \langle x, \delta \rangle \in \Phi \land K_\delta \subseteq B]\}$.
The set $\Phi$ is called a weak $\alpha$-enumeration operator.
\end{defn}

\begin{defn}\label{defn_ae_reducibility}($\alpha$-enumeration reducibility)\\
$A$ is $\alpha$-enumeration reducible to $B$ denoted as $A \le_{\alpha e} B$ iff $\exists W \in \Sigma_1(L_\alpha)$ st
$\forall \gamma < \alpha [ K_\gamma \subseteq A \iff \exists \delta < \alpha [ \langle \gamma, \delta \rangle \in W \land K_\delta \subseteq B]]$.

Denote the fact that $A$ reduces to $B$ via $W$ as $A=W(B)$.
\end{defn}

\begin{fact}(Transitivity)\\
The $\alpha$-enumeration reducibility $\le_{\alpha e}$ is transitive.
But in general the weak $\alpha$-enumeration reducibility is not transitive.
\end{fact}

\begin{lemma}\label{lemma_join_ce_reduction}
$A \le_{\alpha e} B \oplus C \land B \in \Sigma_1(L_\alpha) \implies A \le_{\alpha e} C$
\qed
\end{lemma}

\begin{fact}\label{fact_wae_ae_implies_wae}
If $A \le_{w \alpha e} B$ and $B \le_{\alpha e} C$, then $A \le_{w \alpha e} C$.
\end{fact}

\subsection{Properties of $\alpha$-enumeration operator}
\begin{fact}\label{lemma_enum_from_aeop}
If $A \subseteq \alpha$, then $\aeop{e}(A) \le_{w\alpha e} A$.
\end{fact}

\begin{fact}\label{prop_enum_op_monotonicity}(Monotonicity)\\
$\forall e < \alpha \forall A, B \subseteq \alpha [A \subseteq B \implies \aeop{e}(A) \subseteq \aeop{e}(B)]$.
\qed
\end{fact}

\begin{prop}\label{prop_enum_op_witness}(Witness property)\\
If $x \in \aeop{e}(A)$, then $\exists K \subseteq A [K \in L_\alpha \land x \in \aeop{e}(K)]$.
\end{prop}
\begin{proof}
Note $\aeop{e}(A) := \bigcup \{K_\gamma : \exists \delta < \alpha [\langle \gamma, \delta \rangle \in W_e \land K_\delta \subseteq A\}$.
Thus if $x \in \aeop{e}(A)$, then $\exists \gamma < \alpha$ st $x \in K_\gamma$ and so $\exists \delta < \alpha [\langle \gamma, \delta \rangle \in W_e \land K_\delta \subseteq A]$.
Taking $K$ to be $K_\delta$ concludes the proof.
\end{proof}

\subsection{Totality}
\begin{defn}\footnote{From \cite{chong1984techniques} p8.}
The \emph{computable join} of sets $A, B \subseteq \alpha$ denoted $A \oplus B$ is defined to be

$A \oplus B := \{2a : a \in A\} \cup \{2b+1 : b \in B\}$.
\end{defn}

The computable join satisfies the usual properties of the case $\alpha=\omega$.

The generalization of the Turing reducibility corresponds to two different notions - weak $\alpha$ reducibility and $\alpha$ reducibility.

\begin{defn}(Total reducibilities)
\begin{itemize}
\item $A$ is $\alpha$-reducible to $B$ denoted as $A \le_{\alpha} B$ iff $A \oplus \overline{A} \le_{\alpha e} B \oplus \overline{B}$.
\item $A$ is weakly $\alpha$-reducible to $B$ denoted as $A \le_{w\alpha} B$ iff $A \oplus \overline{A} \le_{w\alpha e} B \oplus \overline{B}$.
\end{itemize}
\end{defn}

\begin{defn}(Total set)\\
A subset $A \subseteq \alpha$ is total iff $A \le_{\alpha e} \overline{A}$ iff $A \equiv_{\alpha e} A \oplus \overline{A}$.
\end{defn}

\subsection{$\Sigma_1$-projectum}

\begin{defn}(Projectum\footnote{Definition 1.19 in \cite{chong1984techniques}.})\\
The \emph{$\Sigma_1$ projectum} (projectum for short) of $\alpha$ is\\
$\alpha^* := \mathrm{min}\{\gamma \le \alpha : \exists A \subseteq \gamma [A \in \Sigma_1(L_\alpha) \land A \not\in L_\alpha]\}$.
\end{defn}

\begin{prop}\label{prop_projectum_eq_defs}\footnote{Theorem 1.20 in \cite{chong1984techniques}.}
The following ordinals are equal:\\
i) $\alpha^* := \mathrm{min}\{\gamma < \alpha : \exists A \subseteq \gamma [A \in \Sigma_1(L_\alpha) \land A \not\in L_\alpha]\}$\\
ii) $\mathrm{min}\{\gamma \le \alpha : \exists $ partial surjection $ p_1: \gamma \rightharpoonup \alpha \in \Sigma_1(L_\alpha)\}$\\
iii) $\mathrm{min}\{\gamma \le \alpha : \exists $ total injection $ i: \alpha \rightarrowtail \gamma \in \Sigma_1(L_\alpha)\}$.
\qed
\end{prop}

\begin{prop}\label{prop_indexing_ace_sets_with_proj}(Indexing $\alpha$-c.e. sets with a projectum)\\
We can index all $\alpha$-c.e. sets just with indices from $\alpha^*$.
\qed
\end{prop}

%Degree theory
\subsection{Degree Theory}

\begin{defn}(Degrees)
\begin{itemize}
\item $\mathcal{D}_\alpha := \powerset{\alpha}/\equiv_\alpha$ is a set of \emph{$\alpha$-degrees}.
\item $\mathcal{D}_{\alpha e} := \powerset{\alpha}/\equiv_{\alpha e}$ is a set of \emph{$\alpha$-enumeration} degrees.
\end{itemize}
Induce $\le$ on $\mathcal{D}_\alpha$ and $\mathcal{D}_{\alpha e}$ by $\le_\alpha$ and $\le_{\alpha e}$ respectively.
\end{defn}

\begin{fact}(Embedding of the total degrees)\\
$\langle \mathcal{D}_\alpha, \le \rangle$ embeds into $\langle \mathcal{D}_{\alpha e}, \le \rangle$ via $\iota:\mathcal{D}_\alpha \hookrightarrow \mathcal{D}_{\alpha e}$, $A \mapsto A \oplus \overline{A}$.
\end{fact}

\begin{defn}(Total degrees)\\
Let $\iota:\mathcal{D}_\alpha \hookrightarrow \mathcal{D}_{\alpha e}$ be the embedding from above.
The total $\alpha$-enumeration degrees $\mathcal{TOT}_{\alpha e}$ are the image of $\iota$, i.e. $\mathcal{TOT}_{\alpha e} := \iota[\mathcal{D}_\alpha]$.
\end{defn}

\begin{prop}\label{cor_unboundedness_of_ae_degs}(Unboundedness of $\alpha$-enumeration degrees)
For every set $A \subseteq \alpha$, there is a set $B \subseteq \alpha$ st $A <_{\alpha e} B$.
\qed
\end{prop}

\subsection{Regularity}
\subsubsection{Regularity and quasiregularity}

\begin{defn}(Regularity and quasiregularity)
\begin{itemize}
\item A subset $A \subseteq \alpha$ is \emph{$\alpha$-regular} iff $\forall \gamma < \alpha. A \cap \gamma \in L_\alpha$.
\item A subset $A \subseteq \alpha$ is \emph{$\alpha$-quasiregular} iff $\forall \gamma < \mathrm{sup}(A). A \cap \gamma \in L_\alpha$.
\end{itemize}
\end{defn}

If clear from the context, we just say \emph{regular} and \emph{quasiregular} instead of $\alpha$-regular and $\alpha$-quasiregular respectively.

\begin{fact}
i) $\forall A \subseteq \alpha [A$ regular $\iff \overline{A}$ regular$]$,\\
ii) $\forall A, B \subseteq \alpha [A$ regular $\land B$ regular $\implies A \oplus B$ regular $]$,\\
iii) $\forall A, B \subseteq \alpha [A$ regular $\land B$ regular $\implies A \cup B$ regular $]$.
\end{fact}

\begin{thm}\label{thm_sacks_regular_set_existence}(Sack's Theorem on regular set existence\footnote{Sacks \cite{sacks1990higher}, theorem 4.2})\\
Let $A$ be $\alpha$-computably enumerable. Then there exists a regular, $\alpha$-computably enumerable $B$ of the same $\alpha$-degree as $A$.
\end{thm}

\begin{thm}\label{thm_shore_splitting}(\emph{Shore's Splitting Theorem} \cite{shore1975splitting})\\
Let $B$ be $\alpha$-computably enumerable and regular. Then there exists $\alpha$-computably enumerable $A_0, A_1$ st $B = A_0 \sqcup A_1$ and
$B \not\le_{\alpha} A_i (i \in \{0,1\})$.
\qed
\end{thm}
\subsubsection{Megaregularity}
Megaregularity of a set $A$ measures the amount of the admissibility of a structure structure $\langle L_\alpha, A \rangle$, i.e. a structure extended by a predicate with an access to $A$.

\begin{note}(Formula with a positive/negative parameter)
\begin{itemize}
\item Let $B$ denote a set, $B^+$ its enumeration, $B^-$ the enumeration of its complement $\overline{B}$.
\item Denote by $\Sigma_1(L_\alpha, B)$ the class of $\Sigma_1$ formulas with a parameter $B$ or in $L_\alpha$.
\item A $\Sigma_1(L_\alpha, B)$ formula $\phi(\overline{x}, B)$ is $\Sigma_1(L_\alpha, B^+)$ iff $B$ occurs in $\phi(\overline{x}, B)$ only positively, i.e. there is no negation up the formula tree above the literal $x \in B$.
\item Similarly, a $\Sigma_1(L_\alpha, B)$ formula $\phi(\overline{x}, B)$ is $\Sigma_1(L_\alpha, B^-)$ iff $B$ occurs in $\phi(\overline{x}, B)$ only negatively.
\end{itemize}
\end{note}

\begin{defn}(Megaregularity)\\
Let $\mathcal{B} \in \{B, B^-, B^+\}$ and add $B$ as a predicate to the language for the structure $\langle L_\alpha, \mathcal{B} \rangle$.
\begin{itemize}
\item Then $\mathcal{B}$ is $\alpha$-\emph{megaregular} iff $\alpha$ is $\Sigma_1(L_\alpha, \mathcal{B})$ admissible iff the structure $\langle L_\alpha, \mathcal{B} \rangle$ is admissible,
i.e. every $\Sigma_1(L_\alpha, \mathcal{B})$ definable function satisfies the replacement axiom:
$\forall f \in \Sigma_1(L_\alpha, \mathcal{B}) \forall K \in L_\alpha. f[K] \in L_\alpha$.
\item $B$ is \emph{positively $\alpha$-megaregular} iff $B^+$ is $\alpha$-megaregular.
\item $B$ is \emph{negatively $\alpha$-megaregular} iff $B^-$ is $\alpha$-megaregular.
\end{itemize}
\end{defn}

If clear from the context, we just say \emph{megaregular} instead of $\alpha$-megaregular.

\begin{rk}(Hyperregularity and megaregularity)\\
A person familiar with the notion of hyperregularity shall note that a set is
megaregular iff it is regular and hyperregular.
\end{rk}

\begin{prop}\label{prop_alpha_fin_iff_bounded_and_alpha_comp_with_oracle}
Let $\mathcal{B} \in \{B, B^-, B^+\}$ be megaregular and let $A \subseteq \alpha$. Then:
$A \in L_\alpha$ iff $A \in \Delta_1(L_\alpha, \mathcal{B})$ and $A$ is bounded by some $\beta < \alpha$.
\end{prop}

\begin{proof}
$\implies$ direction is clear. For the other direction, assume that $A \in \Delta_1(L_\alpha, \mathcal{B})$ and $A \subseteq \beta < \alpha$ for some $\beta$.
WLOG let $A \not=\emptyset$ and let $a \in A$.
Define a function $f:\alpha \to \alpha$ by $f(x)=y :\iff x \not\in \beta \land y =a \lor x \in \beta \land [x \in A \land x=y \lor x \not\in A \land y=a]$.
Since  $A \in \Delta_1(L_\alpha, \mathcal{B})$, the function $f$ is $\Sigma_1(L_\alpha, \mathcal{B})$ definable.
By the megaregularity of $\mathcal{B}$, we have that $A=f[\beta] \in L_\alpha$ as required.
\end{proof}

\begin{cor}\label{cor_mr_deg_invariance}(Megaregularity closure and degree invariance)\\
i) If $A \le_{\alpha e} B$ and $B^+$ megaregular, then $A^+$ megaregular.\\
ii) If $A \equiv_{\alpha e} B$, then $[A^+$ megaregular iff $B^+$ megaregular $]$.\\
iii) If $A \le_{\alpha} B$ and $B$ megaregular, then $A$ megaregular.\\
iv) If $A \equiv_{\alpha} B$, then $[A$ megaregular iff $B$ megaregular $]$.\\
v) If $A \in \Sigma_1(L_\alpha)$, then $A^+$ is megaregular.\\
vi) If $A \in \Delta_1(L_\alpha)$, then $A$ is megaregular.
\qed
\end{cor}

\subsubsection{Regularity and definability}
\begin{prop}\label{prop_correspondence_wae_ae}($\Sigma_1$ definability and $\alpha$-enumeration reducibilities correspondence)\\
We have the following implication diagram:\\
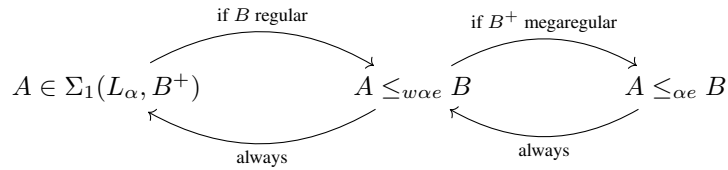
\begin{center}
\begin{tikzcd}
A \in \Sigma_1(L_\alpha, B^+) \arrow[rr, bend left, "\text{if $B$ regular}"] & &
A \le_{w\alpha e} B \arrow[rr, bend left, "\text{if $B^+$ megaregular}"]
\arrow[ll, bend left, "\text{always}"] & &
A \le_{\alpha e} B \arrow[ll, bend left, "\text{always}"]
\end{tikzcd}
\end{center}
\qed
\end{prop}

\subsubsection{Notions of regularity by strength}

\begin{rk}
We have the following \emph{strict} separation of the notions where $\alpha$-finiteness is the strongest condition and quasiregularity is the weakest:\\
$\alpha$-finite $\implies$ $\alpha$-computable $\implies$ megaregular $\implies$ regular $\implies$ quasiregular
\end{rk}

\subsection{Useful lemmas}
%done1, 29.3.2015
\begin{lemma}\label{lemma_sacks_lemma6}\footnote{From lemma 6 in \cite{sacks1963degrees} on p66.}
$A_0 \cap A_1 = \emptyset, A_{i \in \{0,1\}} \in \Sigma_1(L_\alpha, A_0 \sqcup A_1) \implies A_0 \sqcup A_1 \equiv_\alpha A_0 \oplus A_1$.
\end{lemma}

\begin{proof}
$A_0 \sqcup A_1 \le_\alpha A_0 \oplus A_1$ trivially. Let $i \in \{0,1\}$.
For $A_0 \oplus A_1 \le_\alpha A_0 \sqcup A_1$:
$x \in A_i$ recognizable by $A_i \in \Sigma_1(L_\alpha, A_0 \sqcup A_1)$.
Also $x \not\in A_i$ is recognizable since $x \not\in A_i \iff x \in A_{1-i} \lor x \not\in A_0 \sqcup A_1$ by disjointness and both $x \in A_{1-i}, x \not\in A_0 \sqcup A_1$ are recognizable from $A_0 \sqcup A_1$.
Hence $A_0 \sqcup A_1 \equiv_\alpha A_0 \oplus A_1$.
\end{proof}

%done2
The lemma implies that if $A_0$, $A_1$ are disjoint $\alpha$-incomparable $\alpha$-computably enumerable sets, then $A_0 \sqcup A_1 \equiv_\alpha A_0 \oplus A_1$ (proposition 3.3 in \cite{soare1987recursively}).

\begin{lemma}\label{lemma_numbering_of_pairs_of_alpha_fin_sets}
There exists an $\alpha$-computable function $g : \alpha \times \alpha \times \alpha \to \alpha$ st $D_\eta:=\{x | g(\eta,x,1)=1\} \in L_\alpha$, $E_\eta:=\{x | g(\eta,x,2)=1\} \in L_\alpha$ and for every pair $(\hat{D},\hat{E})$ of $\alpha$-finite subsets of $\alpha$ there is an index $\eta < \alpha$ st $D_\eta = \hat{D}$ and $E_\eta = \hat{E}$.

Therefore we can $\alpha$-effectively number the pairs of the $\alpha$-finite subsets of $\alpha$ by the indices of $\alpha$.
\end{lemma}

\begin{proof}
Note that there are $\alpha$-computable bijections $j:\alpha \to L_\alpha$ and $f:\alpha \to \alpha \times \alpha$.
Let $\pi_1$ and $\pi_2$ be the projections.
Define $g(\eta,x,k) := [x \in j \circ \pi_k \circ f(\eta)]$.
Then $g$ is the required $\alpha$-computable function.
\end{proof}

\begin{lemma}\label{lemma_alpha_computable_indexing_of_computed_sets}
Let $i,j,k: \alpha \times \alpha \to \alpha$ be any $\alpha$-computable numberings of $\alpha$-finite subsets of $\alpha$. Then:\\
i) There is an $\alpha$-computable function $u:\alpha \to \alpha$ st\\ $\forall \gamma < \alpha. \bigcup_{x \in j(\gamma)} i(x) = k(u(\gamma))$.\\
ii) There is an $\alpha$-computable function $v:\alpha \times \alpha \to \alpha$ st\\
$\forall \gamma, \delta < \alpha. k(v(\gamma,\delta))=i(\gamma) \oplus j(\delta)$.\\
iii) There exist $\alpha$-computable functions $i_{\pi_1},i_{\pi_2}:\alpha \to \alpha$ st\\
$\forall l \in \{1,2\} \forall \gamma < \alpha. k(i_{\pi_l}(\gamma)) = \{x_l : \langle x_1, x_2 \rangle \in i(\gamma)\}$.\\
iv) There exists an $\alpha$-computable function $i_{p_2} : \alpha \to \alpha$ st\\
$\forall \gamma < \alpha. k(i_{p_2}(\gamma)) = i(\gamma) \times j(\gamma)$.\\
v) There exists an $\alpha$-computable function $w:\alpha \times \alpha \to \alpha$ st if $\gamma, \delta < \alpha$, then $k(w(\gamma, \delta)) = \{\langle x, y \rangle : x \in j(\delta) \land y \in j(\gamma) \land y \in i(x)\}$.\\
vi) There exists a function $t_{i,j}:\alpha \to \alpha \in \Sigma_1(L_\alpha)$ st $\forall \gamma < \alpha. i(\gamma)=j(t_{i,j}(\gamma))$.\\
vii) Let $K(\gamma) := \bigcup_{x \in j(\gamma)} i(x)$.
Then there exists a function $s_{i,j}:\alpha \to \alpha \in \Sigma_1(L_\alpha)$ st
$\forall \gamma < \alpha. s_{i,j}(\gamma)=$
$\begin{cases}
0 & K(\gamma)=\emptyset\\
\mathtt{sup}(K(\gamma)) & K(\gamma)\not=\emptyset
\end{cases}$.
\qed
\end{lemma}

%%%%%%%%%%%%%%%%%%%%%%%%%%%%%%%%%%%%%%%%%%%%%%%%%%%%%%%%%%%%%%%%%%%%%%%%%%%%%
\section{Semicomputability}\label{section_semicomputability}%%%%%%%%%%%%%%%%%
%%%%%%%%%%%%%%%%%%%%%%%%%%%%%%%%%%%%%%%%%%%%%%%%%%%%%%%%%%%%%%%%%%%%%%%%%%%%%

The goal of this section is to lift the necessary results of Jockusch \cite{jockusch1968semirecursive} on semicomputable sets from the level $\omega$ to a level $\alpha$.

\begin{defn}A set $A \subseteq \alpha$ is \emph{$\alpha$-semicomputable} iff there exists a total $\alpha$-computable function $s_A:\alpha \times \alpha \to \alpha$ called a \emph{selector function} satisfying:\\
i)$\forall x, y \in \alpha. s_A(x, y) \in \{x, y\}$,\\
ii)$\forall x, y \in \alpha [\{x ,y\} \cap A \not= \emptyset \implies s_A(x,y) \in A]$.\\
Denote by $\asc$ the class of $\alpha$-semicomputable sets.
\end{defn}

\begin{fact}(Semicomputability closure)\\
%i) $A \subseteq B \land B \in \asc \implies A \in \asc$,\\
i) $A \in \asc \iff \overline{A} \in \asc$,\\
ii) $A \oplus B \in \asc \implies A \in \asc \land B \in \asc$.
%iii) $A \cup B \in \asc \implies A \oplus B \in \asc$.
\end{fact}

\begin{defn}(Index set)\\
An \emph{index set} for a set $A \subseteq \alpha$ denoted as $A_I$ is a set of \emph{all} indices of $\alpha$-finite subsets of $A$, i.e. $A_I := \{\gamma < \alpha : K_\gamma \subseteq A\}$. 
\end{defn}

\begin{prop}(Semicomputability of an index set)\\
For every set $A \subseteq \alpha$, its index set $A_I$ is $\alpha$-semicomputable.
\end{prop}
\begin{proof}
Define the selector function of $A_I$ as $s_{A_I} := \{ \langle \gamma, \delta, \rangle : K_\gamma \subseteq K_\delta \}$. The function $s_{A_I}$ is $\alpha$-computable as required.
\end{proof}

\begin{defn}\label{defn_alpha_rationals_order}(Binary ordering)\\
Define $<_b \subseteq \mathcal{P}(\alpha) \times \mathcal{P}(\alpha)$ and $\le_b \subseteq \mathcal{P}(\alpha) \times \mathcal{P}(\alpha)$ to be numerical orderings on the binary representation of the compared sets:
\begin{itemize}
\item $A <_b B :\iff \exists \beta \in \alpha [\beta \not\in A \land \beta \in B \land A \cap \beta = B \cap \beta]$,
\item $A \le_b B :\iff A <_b B \lor A = B$.
\end{itemize}
\end{defn}

%done2
\begin{rk}\label{rk_b_ordering_computable}
The restrictions of the orderings $<_b$ and $\le_b$ to $\alpha$-finite sets are first-order definable and $\alpha$-computable since an $\alpha$-finite set is bounded.
\end{rk}

\begin{prop}(Properties of binary ordering)\\
Let $\lhd \in \{<, \le\} $, then:\\
i) $<_b$ is a strict total order,\\
ii) $\le_b$ is a total order,\\
iii) $([0,1], \lhd_{\hreals}) \cong (\mathcal{P}(\alpha), \lhd_b))$ where $\hreals$ is an appropriate model of the hyperreal numbers,\\
iv) $A \lhd_b B \iff \overline{B} \lhd_b \overline{A}$.
\end{prop}

\begin{proof}
i), ii), iii) are trivial. To prove iv), use iii) and consider $\mathcal{P}(\alpha)$ as the interval $[0,1]$ from the field of hyperreals, where $0 := \emptyset$ and $1 := \alpha$. Then:
$\overline{B} \lhd_b \overline{A} \iff 1 - B \lhd_b 1 - A \iff -B \lhd_b -A \iff A \lhd_b B$.
\end{proof}

\begin{fact}(Binary and subset ordering)\\
i) $A \subset B \implies A <_b B$,\\
ii) $A \subseteq B \implies A \le_b B$,\\
iii) $A = B \iff A \equiv_b B$.
\end{fact}

\begin{note}
If $A \le_b C$ and $B \le_b C$, is it true that $A \cup B \le_b C$?

No. Consider $A=011..., B=100..., C=110...$. Then $A \cup B=111...$. Thus $A \le_b C$ and $B \le_b C$, but $\neg A \cup B \le_b C$.
\end{note}

\begin{defn}\label{defn_sc_segments}
Given a set $A$ define $L_A := \{x \in \alpha : K_x \le_b A \}, R_A := \overline{L_A}$.
\end{defn}

\begin{rk}\label{rk_la_ra_def}
If $A \not\in L_\alpha$, then:
\begin{itemize}
\item $L_A = \{ x < \alpha : K_x <_b A \}$ are $\alpha$-finite sets \emph{left} of $A$,
\item $R_A = \{ x < \alpha : A <_b K_x \}$ are $\alpha$-finite sets \emph{right} of $A$.
\end{itemize}
\end{rk}

\begin{fact}\label{fact_properties_of_left_right_af_sets}(Properties of left/right $\alpha$-finite sets)\\
Let $A \subseteq \alpha$ and $\beta, \gamma, \delta < \alpha$. Then:\\
i)  $K \in L_\alpha \land K_\delta = \bigcup_{\gamma \in K} K_\gamma \land \delta \in L_A \implies K \subseteq L_A$,\\
ii) $\beta \in L_A \land \gamma \in R_A \land K_\beta \cap \delta = K_\gamma \cap \delta \implies K_\beta \cap \delta \subseteq A$.
\end{fact}

\begin{lemma}\label{lemma_la_ra_are_sc}
For any $A \subseteq \alpha$ the sets $L_A, R_A$ are $\alpha$-semicomputable.
\end{lemma}

\begin{proof}
$L_A$ is $\alpha$-semicomputable since it has an $\alpha$-computable selector function
$s:= \{(x,y) : K_x \le_b K_y\} \cup \{(y,x) : K_x >_b K_y\}$ by remark \ref{rk_b_ordering_computable}.
\end{proof}

\begin{lemma}\label{lemma_sc_correspondence}
Let $A \subseteq \alpha$ be a quasiregular set, then $A \equiv_\alpha L_A \equiv_\alpha R_A$.
\end{lemma}

\begin{proof}
If $A \in \Delta_1(L_\alpha)$, then trivially $A \equiv_\alpha L_A \equiv_\alpha R_A$. Hence WLOG assume that $A \not\in L_\alpha$ and use \cref{rk_la_ra_def}. Also WLOG $A \not\in \Delta_1(L_\alpha)$ and so in the proof implicitly use the property:
$\forall x \in A \exists y, z [x < y < \alpha \land x < z < \alpha \land y \not\in A \land z \in A]$.

Note that $\bigcup_{x \in K_\gamma} K_x \in L_\alpha$. Hence for any $\gamma < \alpha$ we have:
$K_\gamma \subseteq L_A \iff \exists \beta < \alpha [K_\beta <_b A \land \forall x \in K_\gamma. K_x <_b K_\beta]$.
Thus $L_A \le_{\alpha e} A$ via $W := \{ \langle \gamma, \delta \rangle : \exists \beta < \alpha [K_\delta = \{\beta\} \land \forall x \in K_\gamma. K_x <_b K_\beta]\} \in \Sigma_1(L_\alpha)$. By symmetry $R_A \le_{\alpha e} A$.
Hence $L_A \oplus R_A \le_{\alpha e} A$.

Let $\widehat{A}$ denote $A$ or $\overline{A}$. Then $K_\gamma \subseteq \widehat{A} \iff \exists \beta_L, \beta_R < \alpha [ \forall x \in K_\gamma \forall y \le x [ y \in K_{\beta_L} \iff y \in K_{\beta_R}] \land K_\gamma \subseteq \widehat{K_{\beta_L}} \land \beta_L \in L_A \land \beta_R \in R_A]$ for any $\gamma < \alpha$ using the quasiregularity of $A$ and \cref{fact_properties_of_left_right_af_sets}ii.
Hence define $W := \{ \langle \gamma, \delta \rangle : \exists \beta_L, \beta_R < \alpha [ \forall x \in K_\gamma \forall y \le x [ y \in K_{\beta_L} \iff y \in K_{\beta_R}] \land K_\gamma \subseteq \widehat{K_{\beta_L}} \land K_\delta = \{\beta_L\} \oplus \{\beta_R\}] \}$.
Note that $W \in \Sigma_1(L_\alpha)$ and so $\widehat{A} \le_{\alpha e} L_A \oplus R_A$ via $W$.
Hence $A \oplus \overline{A} \le_{\alpha e} L_A \oplus R_A$.

Therefore $A \oplus \overline{A} \equiv_{\alpha e} L_A \oplus R_A = L_A \oplus \overline{L_A} = R_A \oplus \overline{R_A}$ and so $A \equiv_\alpha L_A \equiv_\alpha R_A$ as required.
\end{proof}

\begin{lemma}\label{lemma_jockusch_strongly_non-re_semicomputable}\footnote{Adapted from Lemma 5.5 in \cite{jockusch1968semirecursive} for $\alpha=\omega$}
$B \in \Sigma_1(L_\alpha) \land B >_\alpha 0 \implies$\\
$\exists A [A$ regular $\land A \equiv_\alpha B \land L_A \not\in \Pi_1(L_\alpha) \land L_A \not\in \Sigma_1(L_\alpha)]$.
\end{lemma} 

\begin{proof}
By \cref{thm_sacks_regular_set_existence} every $\Sigma_1(L_\alpha)$ set is $\alpha$-equivalent to some regular set, so WLOG assume that $B$ is regular.
By Shore's Splitting Theorem \ref{thm_shore_splitting}, $\exists C_0, D_0 \in \Sigma_1(L_\alpha)[ B = C_0 \sqcup D_0 \land C_0 |_\alpha D_0$ (incomparable wrt $\alpha$-reducibility) $]$.
Using \cref{thm_sacks_regular_set_existence} again, let $C, D$ be $\alpha$-c.e. regular sets st $C \equiv_\alpha C_0$ and $D \equiv_\alpha D_0$.
Define $A := C \oplus \overline{D}$.

Note $A = C \oplus \overline{D} \equiv_\alpha C_0 \oplus \overline{D_0}$.
Hence $A \equiv_\alpha B$ by \cref{lemma_sacks_lemma6} as required.

As $D$ is regular, so $\overline{D}$ is regular.
As $C$ and $\overline{D}$ are regular, so $A = C \oplus \overline{D}$ is regular as required.

Next we prove $L_A \not \in \Pi_1(L_\alpha) \land L_A \not\in \Sigma_1(L_\alpha)$.
For suppose to the contrary that $\neg(L_A \not\in \Pi_1(L_\alpha) \land L_A \not\in \Sigma_1(L_\alpha))$. Then $L_A \in \Sigma_1(L_\alpha) \lor L_A \in \Pi_1(L_\alpha)$.
\begin{itemize}
\item Case $L_A \in \Sigma_1(L_\alpha)$:\\
Note that $\overline{D} \le_{\alpha e} C \oplus \overline{C}$ via\\
$W := \{\langle \gamma, \delta \rangle : \beta = \mathrm{min}\{ \epsilon < \alpha : K_\gamma \cap \epsilon = K_\gamma\} \land \exists \zeta \in L_A \forall x < \beta [$\\
$(2x \in K_\delta \iff 2x+1 \not\in K_\delta \iff 2x \in K_\zeta) \land$\\
$(x \in K_\gamma \implies 2x+1 \in K_\zeta) \land$\\
$(2x+1 \not\in K_\zeta \implies x \in D)]\}$.
The set $W$ is $\alpha$-c.e. since $L_A$ and $D$ are $\alpha$-c.e.
The condition $2x \in K_\delta \iff 2x+1 \not\in K_\delta$ ensures that $K_\delta$ contains the initial segment $C \cap \beta$ of $C$.
The conditions $2x \in K_\delta \iff 2x \in K_\zeta$ and $2x+1 \not\in K_\zeta \implies x \in D$ ensure that $K_\zeta$ contains the initial segment $(C \cap \beta) \oplus (\overline{D} \cap \beta)$ of $C \oplus \overline{D}$.
Finally, the condition $x \in K_\gamma \implies 2x+1 \in K_\zeta$ verifies that $K_\gamma$ is a subset of $\overline{D}$, or more precisely a subset of its initial segment $\overline{D} \cap \beta$.

As $D$ is $\alpha$-c.e., so this gives us $D \le_\alpha C$ which is a contradiction to the case $L_A \in \Sigma_1(L_\alpha)$.

\item Case $L_A \in \Pi_1(L_\alpha)$:\\
Note that $R_A = \overline{L_A} \in \Sigma_1(L_\alpha)$.
Hence similarly $C \le_\alpha D$ using the fact that $R_A$ and $C$ are both $\alpha$-c.e. by applying a symmetric argument to the one above.
This is a contradiction to the case $L_A \in \Pi_1(L_\alpha)$.
\end{itemize}

So by the two cases $L_A \not\in \Pi_1(L_\alpha) \land L_A \not\in \Sigma_1(L_\alpha)$.

Therefore given $B >_\alpha 0$, there is a regular set $A$ st $A \equiv_\alpha B \land L_A \not\in \Pi_1(L_\alpha) \land L_A \not\in \Sigma_1(L_\alpha)$ as required.
\end{proof}

\begin{thm}\label{thm_jockusch_strongly_non-re_semicomputable}
Let $B \subseteq \alpha$ be quasiregular and $B >_\alpha 0$. Then there exists an $\alpha$-semicomputable set $A$ st
$A \equiv_\alpha B \land A \not\in \Sigma_1(L_\alpha) \land A \not\in \Pi_1(L_\alpha)$.
\end{thm}

\begin{proof}
If $\adeg{B}$ is $\alpha$-c.e. degree, then WLOG let $B \in \Sigma_1(L_\alpha)$. Then by \cref{lemma_jockusch_strongly_non-re_semicomputable} there is $C$ st $C$ is quasiregular, $B \equiv_\alpha C \land L_C \not\in \Sigma_1(L_\alpha) \land L_C \not\in \Pi_1(L_\alpha)$.
By \cref{lemma_sc_correspondence} and quasiregularity of $C$ we have that $C \equiv_\alpha L_C$ and so $B \equiv_\alpha L_C$.
Hence $A := L_C$ is the required $\alpha$-semicomputable set by \cref{lemma_la_ra_are_sc}.

Otherwise $\adeg{B}$ is not an $\alpha$-c.e. degree and so $\forall C \in \adeg{B} [C \not\in \Sigma_1(L_\alpha) \land C \not\in \Pi_1(L_\alpha)]$. Note that $A := L_B \equiv_\alpha B$ by the quasiregularity of $B$ and by \cref{lemma_sc_correspondence} and so $A \not\in \Sigma_1(L_\alpha) \land A \not\in \Pi_1(L_\alpha)$. Finally, $A$ is $\alpha$-semicomputable by \cref{lemma_la_ra_are_sc} as required.
\end{proof}

%%%%%%%%%%%%%%%%%%%%%%%%%%%%%%%
\section{Kalimullin pair}%%%%%%%%%%%%%%%%%%%
%%%%%%%%%%%%%%%%%%%%%%%%%%%%%%%
The goal of this section is to generalize the results of Kalimullin \cite{sh2003definability} on the definability of a Kalimullin pair to a level $\alpha$.

\subsection{Introduction and basic properties}
\begin{defn}Sets $A, B \subseteq \alpha$ are a \emph{$\alpha$-$U$-Kalimullin pair} denoted by $\ukpair{U}{A}{B}$ iff $\exists W \le_{\alpha e} U [A \times B \subseteq W \land \overline{A} \times \overline{B} \subseteq \overline{W} ]$. If clear, we omit the prefix $\alpha$ and say $U$-Kalimullin pair (or just $U$-$\mathcal{K}$-pair) and denote it by $\ukpair{U}{A}{B}$.
Similarly, if $U \in \Sigma_1(L_\alpha)$, then we say that $A,B$ are a Kalimullin pair (or just $\mathcal{K}$-pair) and denote $\kpair{A}{B}$.

The set $W$ is called a \emph{witness} to the U-Kalimullin pair.
\end{defn}

\begin{prop}\label{prop_trivial_u_k_pair}\footnote{Proposition 2.2 in \cite{sh2003definability} for $\alpha=\omega$.}
If $A \le_{\alpha e} U$, then $\forall B \subseteq \alpha. \ukpair{U}{A}{B}$.
\end{prop}

\begin{proof}
Take the witness $W := A \times \alpha$.
\end{proof}

\begin{prop}\label{prop_sc_pair_kalimullin}
If $A$ is $\alpha$-semicomputable, then $\kpair{A}{\overline{A}}$.
\end{prop}

\begin{proof}
Define the witness $W \in \Sigma_1(L_\alpha)$ to the Kalimullin pair $\kpair{A}{\overline{A}}$ to be
$W := \{ (x,y) \in \alpha : s_A(x,y)=x\}$ where $s_A$ is an $\alpha$-computable selector function for an $\alpha$-semicomputable set $A$.
\end{proof}

\begin{defn}$A, B \subseteq \alpha$ are a \emph{trivial} Kalimullin pair iff $\kpair{A}{B}$ and $A \in \Sigma_1(L_\alpha) \lor B \in \Sigma_1(L_\alpha)$. 
If $A, B$ are a not a trivial Kalimullin pair, they form a \emph{nontrivial} Kalimullin pair, denoted by $\kpairnt{A}{B}$.
\end{defn}

\begin{defn}(Maximal Kalimullin pair)\\
A Kalimullin pair $\kpair{A}{B}$ is \emph{maximal} denoted by $\kpairmax{A}{B}$ iff\\
$\forall C, D [A \le_{\alpha e} C \land B \le_{\alpha e} D \land \kpair{C}{D} \implies A \equiv_{\alpha e} C \land B \equiv_{\alpha e} D]$.
\end{defn}

\begin{rk}
Note that in the definition of a maximal Kalimullin pair we use $\alpha$-enumeration reducibility instead of a weak $\alpha$-enumeration reducibility since we want that a maximal Kalimullin pair is definable (given that a Kalimullin pair is definable) in the structure $\langle \mathcal{D}_{\alpha e}, \le \rangle$ where $\le$ is induced by $\le_{\alpha e}$.
\end{rk}

%done2, 26.3.2015
\begin{prop}\label{prop_ch2_cai1.8}
\footnote{From \cite{sh2003definability} and proposition 1.8 in \cite{caidefining}.}
Assume $A , B \subseteq \alpha \land A \not\in \Sigma_1(L_\alpha) \land B \not\in \Sigma_1(L_\alpha) \land \kpair{A}{B}$ where the witness of $\kpair{A}{B}$ is $W$. Then:

i) $A = \{a:\exists b[ b \not\in B \land (a,b) \in W]\}$.

ii) $B = \{b :\exists a[ a \not\in A \land (a,b) \in W]\}$.
\end{prop}

\begin{proof}
i):\\
0. Assume $A , B \subseteq \alpha \land A \not\in \Sigma_1(L_\alpha) \land B \not\in \Sigma_1(L_\alpha) \land \kpair{A}{B}$.\\
1. Define $A_2 := \{a:\exists b [ b \not\in B \land (a,b) \in W]\}$.\\
2. Assume $a \in A$.\\
3. Assume $a \not\in A_2$.\\
4. $\forall b [\neg(b \not\in B \land (a,b) \in W)]$ by 3.\\
5. $\forall b [b \in B \lor (a,b) \not\in W]$ by 4.\\
6. $\forall b [(a,b) \in W \implies b \in B]$ by 5.\\
7. $A \times B \subseteq W$ by 0.\\
8. $B = \{b: \exists a \in \alpha. (a,b) \in W\}$ by 7.\\
9. $B \le_{\alpha e} W$ by 8.\\
10. $W \in \Sigma_1(L_\alpha)$ by 0.\\
11. $B \in \Sigma_1(L_\alpha)$ by 9, 10.\\
12. $\mathrm{false}$ by 0, 11.\\
13. $a \in A_2$ by 3, 12.\\
14. $A \subseteq A_2$ by 2, 13.\\
15. Assume $a \in A_2$.\\
16. $\exists b [b \not\in B \land (a,b) \in W]$ by 1, 15.\\
17. Assume $a \not\in A$.\\
18. $b \not \in B \land a \not\in A$ by 16, 17.\\
19. $\overline A \times \overline B \subseteq \overline W$ by 0.\\
20. $(a,b) \in \overline{W}$ by 18, 19.\\
21. $\mathrm{false}$ by 16, 20.\\
22. $a \in A$ by 17, 21.\\
23. $A_2 \subseteq A$ by 15, 22.\\
24. $A=A_2$ by 14, 23. QED of i).

The proof of ii) is symmetric.
\end{proof}

\begin{cor}\label{prop_ch2_cai1.8_cor}
Assume $A , B \subseteq \alpha \land A \not\in \Sigma_1(L_\alpha) \land B \not\in \Sigma_1(L_\alpha) \land \kpair{A}{B}$. Then:\\
i) $A \le_{w\alpha e} \overline{B}$ and $B \le_{w\alpha e} \overline{A}$,\\
ii) $A \le_{\alpha e} \overline{B}$ if $B^-$ is megaregular, $B \le_{\alpha e} \overline{A}$ if $A^-$ is megaregular.
\end{cor}
\begin{proof}
Follows from \cref{prop_ch2_cai1.8}.
\end{proof}

%done1
%tempting: generalize this for joins and products of an arbitrary even transfinite size.
\begin{lemma}\label{lemma_k_pair_distributivity}(Kalimullin pair distributivity)\\
Suppose that $\bigwedge_{i \in 2} A_i \not=\emptyset$. Then

$\bigwedge_{i \in 2} \kpair{A_i}{B} \iff \kpair{\bigoplus_{i \in 2} A_i}{B} \iff \kpair{\prod_{i \in 2} A_i}{B}$
\end{lemma}

\begin{proof}
Suppose $\bigwedge_{i \in 2} \kpair{A_i}{B}$.
For any $i \in 2$ let
$A_i \times B \subseteq U_i \in \Sigma_1(L_\alpha)$ and $\overline A_i \times \overline B \subseteq \overline U_i$.\\
Define $V := \{(2a+i, b) : (a,b) \in U_i, i \in 2 \}$.\\
Define $W := \{((a_0, a_1),b) : \forall i \in 2. (2a_i+i,b) \in V\}$.\\
Define $U^*_i := \{(a_i,b) : \exists (a_0,a_1). ((a_0, a_1),b) \in W\}$.\\
Then $\bigwedge_{i \in 2} A_i \times B \subseteq U_i \in \Sigma_1(L_\alpha) \land \overline A_i \times \overline B \subseteq \overline U_i \implies$\\
$(\bigoplus_{i \in 2} A_i) \times B \subseteq V \in \Sigma_1(L_\alpha) \land \overline{\bigoplus_{i \in 2} A_i} \times \overline B \subseteq \overline V \implies$\\
$(\prod_{i \in 2} A_i) \times B \subseteq W \in \Sigma_1(L_\alpha) \land \overline{\prod_{i \in 2} A_i} \times \overline B \subseteq \overline W \implies$ (by $\bigwedge_{i \in 2} A_i \not=\emptyset$)\\
$\bigwedge_{i \in 2} A_i \times B \subseteq U^*_i \in \Sigma_1(L_\alpha) \land \overline A_i \times \overline B \subseteq \overline U^*_i$.
\end{proof}

%done2, 12.3.2015
\begin{lemma}\label{lemma_max_k_trivial}
$\kpairmax{A}{B} \implies \kpairnt{A}{B}$
\end{lemma}

\begin{proof}(Of \cref{lemma_max_k_trivial})\\
1. Assume $\kpairmax{A}{B}$.\\
2. Assume $B \equiv_{\alpha e} 0$.\\
3. Assume $A \equiv_{\alpha e} 0$.\\
4. $\exists C \subseteq \alpha \land C >_{\alpha e} A$ by unboudedness of $\mathcal{D}_{\alpha e}$ \ref{cor_unboundedness_of_ae_degs}.\\
5. $L_C \oplus R_C \equiv_{\alpha e} C$ by definition \ref{defn_sc_segments}.\\
6. Assume WLOG $L_C >_{\alpha e} A$ by 4, 5.\\
7. $\kpairmax{L_C}{R_C}$ by $L_C$ $\alpha$-semicomputable and $R_C = \overline{L_A}$.\\
8. $L_C \ge_{\alpha e} A \land R_C \ge_{\alpha e} B$ by 2, 3.\\
9. $L_C \equiv_{\alpha e} A$ by 1, 8.\\
10. $\mathrm{false}$ by 6, 9.\\
11. $A >_{\alpha e} 0$ by 3, 10.\\
12. $K_A := \{x \in \alpha : \Phi^A_x(x) \downarrow\}$.\\
13. $\overline{K_A} \equiv_{\alpha e} K_A \oplus \overline{K}_A$ by 12.\\
14. $A \equiv_{\alpha e} K_A$ by 12.\\
15. $\overline{K_A} >_{\alpha e} K_A$ by 13, 14.\\
16. $\overline{K_A} >_{\alpha e} A$ by 14, 15.\\
17. $\kpair{\overline{K_A}}{B}$ by 2.\\
18. $\overline{K_A} \equiv_{\alpha e} A$ by 1, 17.\\
19. $\mathrm{false}$ by 16, 18.\\
20. $B >_{\alpha e} 0$ by 2, 19.\\
21. $\kpairnt{A}{B}$ by 11, 20.
\end{proof}

%%%%%%%%%%%%%%%%%%%%%%%%%%%%%%%%%%%%%%%
\subsection{Definability of an $\alpha$-Kalimullin pair}\label{section_k_pair_definability}%%%%%%%%%%%%%%
%%%%%%%%%%%%%%%%%%%%%%%%%%%%%%%%%%%%%%%

For this section let $D_x, E_x$ be a pair of $\alpha$-finite sets indexed by $x < \alpha$ according to \cref{lemma_numbering_of_pairs_of_alpha_fin_sets}.
For any $x < \alpha$ define
\begin{gather*}
V_x := \{y < \alpha : D_x \subseteq D_y \land E_x \subseteq E_y\}
\end{gather*}

\begin{lemma}\label{lemma_y_ae_from_x_join_a}
Assume that $x \in Y \iff x \in X \land D_x \subseteq A$ where $D_x$ is an $\alpha$-finite set with an $\alpha$-computable index $x$. Then $Y \le_{\alpha e} X \oplus A$.
\end{lemma}

\begin{proof}
Recall $Y \le_{\alpha e} X \oplus A \iff \exists W \in \Sigma_1(L_\alpha) \forall \gamma < \alpha [K_\gamma \subseteq Y \iff \exists \langle \gamma, \delta\rangle \in W. K_\delta \subseteq X \oplus A]$.
Note $K_\gamma \subseteq Y \iff \forall x \in K_\gamma. x \in Y \iff \forall x \in K_\gamma [x\in X \land D_x \subseteq A] \iff K_\gamma \subseteq X \land \bigcup_{x \in K_\gamma} D_x \subseteq A \iff$
(By \cref{lemma_alpha_computable_indexing_of_computed_sets}i)
$K_\gamma \subseteq X \land K_{u(\gamma)} \subseteq A \iff$
(By \cref{lemma_alpha_computable_indexing_of_computed_sets}ii)
$K_{v(\gamma, u(\gamma))} \subseteq X \oplus A$.
Hence define $W := \{\langle \gamma, \delta\rangle < \alpha : \delta=v(\gamma, u(\gamma)) \}$.
As $u,v \in \Sigma_1(L_\alpha)$, so $W \in \Sigma_1(L_\alpha)$.
Moreover, $K_\gamma \subseteq Y \iff \exists \langle \gamma, \delta \rangle \in W. K_\delta \subseteq X \oplus A$.
Therefore $Y \le_{\alpha e} X \oplus A$.
\end{proof}

\begin{lemma}\label{lemma2_thm_kalimullin_thm2.5}
Assume $M_s \in \Sigma_1(L_\alpha)$ and $X_s \in L_\alpha$.
Let $W := \{\langle a,b \rangle : \exists x \in M_s [a \in D_x \land b \in E_x \land x \in \aeop{e}((X_s \cup (M_s \cap V_x)) \oplus U)]\}$.
Assume $U^+$ is megaregular.
Then $W \le_{\alpha e} U$.
\end{lemma}
\begin{proof}
Let $S_e:=\aeop{e}((X_s \cup (M_s \cap V_x)) \oplus U)$.
We first prove $W \le_{\alpha e} S_e$.
Note $K_\gamma \subseteq W \iff$
$\forall \langle a,b \rangle \in K_\gamma. \langle a, b \rangle \in W \iff$
$\forall \langle a,b \rangle \in K_\gamma. \exists x \in M_s [a \in D_x \land b \in E_x \land x \in S_e] \iff$
$\forall \langle a,b \rangle \in K_\gamma. \exists x \in M_s[\langle a,b \rangle \in P_x \land x \in S_e]$ where $i_P:\alpha \to \alpha \in \Sigma_1(L_\alpha)$ is a function of \cref{lemma_alpha_computable_indexing_of_computed_sets}iv and $P_x := K_{i_P(x)}$.
Define $\phi(\gamma, \delta) :\iff \forall y \in K_\gamma \exists x \in K_\delta. y\in P_x$.
Define $V := \{\langle \gamma, \delta \rangle : K_\delta \subseteq M_s \land \phi(\gamma, \delta)\}$.
Then continuing $K_\gamma \subseteq W \iff$
$\forall y \in K_\gamma \exists x \in M_s [y \in P_x \land x \in S_e] \iff$
$\exists \delta [K_\delta \subseteq M_s \land K_\delta \subseteq S_e \land \phi(\gamma, \delta)] $
(Where $K_\delta \in L_\alpha$ has to exist as an image of an $\alpha$-computable function restricted to an $K_\gamma \in L_\alpha$ by the admissibility of $\alpha$.)
$\iff$
$\exists \langle \gamma, \delta \rangle \in V. K_\delta \subseteq S_e$.
Note $\phi(\gamma,\delta) \iff \exists H [H=w(\gamma, \delta) \land \forall y \in K_\gamma \exists x \in K_\delta. \langle x, y \rangle \in H]$ where $w:\alpha \times \alpha \to \alpha \in \Sigma_1(L_\alpha)$ with $K_{w(\gamma, \delta)} := \{\langle x, y \rangle : x \in K_\delta \land y \in K_\gamma \land y \in P_x\}$ is a function of \cref{lemma_alpha_computable_indexing_of_computed_sets}v.
Hence $\phi(\gamma, \delta) \in \Sigma_1(L_\alpha)$.
As $M_s \in \Delta_1(L_\alpha)$ by $M_s \in \Delta_1(L_\alpha)$, so $V \in \Sigma_1(L_\alpha)$.
Therefore $W \le_{w\alpha e} S_e$.

Note $V_x \in \Sigma_1(L_\alpha)$. By the assumptions $M_s \in \Sigma_1(L_\alpha)$ and $X_s \in L_s$ it is true that $M_s \in \Sigma_1(L_\alpha)$ and $X_s \in \Sigma_1(L_\alpha)$.
Thus $(X_s \cup (M_s \cap V_x)) \in \Sigma_1(L_\alpha)$.
Hence $S_e \le_{w\alpha e} (X_s \cup (M_s \cap V_x)) \oplus U \le_{\alpha e} U$ by \cref{lemma_enum_from_aeop} and \cref{lemma_join_ce_reduction} respectively. Hence $S_e \le_{w\alpha e} U$ by \cref{fact_wae_ae_implies_wae}.

As $U^+$ is megaregular, so $S_e \le_{\alpha e} U$ by \cref{prop_correspondence_wae_ae}.
Hence $W \le_{w\alpha e} U$ by \cref{fact_wae_ae_implies_wae}.
Finally, $W \le_{\alpha e} U$ by the megaregularity of $U^+$ again.
\end{proof}

\begin{lemma}\label{lemma3_thm_intersection_of_vx_alpha_fin}
Let $I \in L_\alpha$. Then exists an index $z < \alpha$  which is $\alpha$-computable from $I$ st $V_z = \bigcap_{x \in I} V_x$.
\end{lemma}

\begin{proof}
Define $f(I)=z :\iff D_z = \bigcup_{x \in I} D_x \land E_z = \bigcup_{x \in I} E_x$.
By \cref{lemma_alpha_computable_indexing_of_computed_sets}i the function $f$ is total and $\alpha$-computable.
Also $\bigcap_{x \in I} V_x = \{y < \alpha : \bigcup_{x \in I} D_x \subseteq D_y \land \bigcup_{x \in I} E_x \subseteq E_y\} = V_{f(I)}=V_z$ as required.
\end{proof}

\begin{lemma}\label{lemma4_thm_z_unbounded}
Let $D \subseteq A \subseteq \alpha$ and $E \subseteq B \subseteq \alpha$ satisfying $A, B \not\in \Sigma_1(L_\alpha)$ and $D, E \in L_\alpha$.
Define $Z:=Z_{D,E}:=\{ x < \alpha : D \subseteq D_x \subseteq A \land E \subseteq E_x \subseteq B \}$.
Then:\\
i) $Z \equiv_{\alpha e} A \oplus B$,\\
ii) $\overline{Z} \le_{w\alpha e} \overline{A} \oplus \overline{B}$,\\
iii) $Z \not\in \Sigma_1(L_\alpha)$,\\
iv) $Z$ is unbounded if $A \oplus B$ is megaregular.
\end{lemma}

\begin{proof}
i)
First note that for all $\alpha$-finite sets $K_\gamma, K_\delta$ there is some $x < \alpha$ st $D_x = K_\gamma, E_x = K_\delta$. Hence if we require that $D_x$ (or $E_x$) is fixed to some $\alpha$-finite set $K \in L_\alpha$, still the remaining sets $E_x$ (or $D_x$) include all $\alpha$-finite sets.
Note $A \le_{\alpha e} Z$ via $W:=\{\langle \gamma, \delta \rangle : \exists x < \alpha [D \cup K_\gamma \subseteq D_x \land K_\delta = \{x\} ]\} \in \Sigma_1(L_\alpha)$. Similarly, $B \le_{\alpha e} Z$.
Consequently, $A \oplus B \le_{\alpha e} Z$.
Define $I_{D,A} := \{x < \alpha : D \subseteq D_x \subseteq A\}$.
Define $I_{E,B} := \{x < \alpha : E \subseteq E_x \subseteq B\}$.
$I_{D,A} \le_{\alpha e} A$ via $W_A := \{ \langle \gamma, \delta \rangle : \forall x \in K_\gamma. D \subseteq D_x \land \bigcup_{x \in K_\gamma} D_x = K_\delta\} \in \Sigma_1(L_\alpha)$.
Similarly $I_{E,B} \le_{\alpha e} B$.
Note that $Z = I_{D,A} \cap I_{E,B}$.
Thus $Z \le_{\alpha e} I_{D,A} \oplus I_{E,B} \le_{\alpha e} A \oplus B$.
Therefore $A \oplus B \equiv_{\alpha e} Z$.

ii) Note that $\overline{I_{D,A}} \le_{w\alpha e} \overline{A}$ via $\Phi^{vw} := \{\langle x, \delta \rangle : \exists y < \alpha [y \not\in D_x \land y \in D \land K_\delta = \emptyset \lor y \in D_x \land K_\delta = \{y\}\}$.
Similarly, $\overline{I_{E,B}} \le_{w\alpha e} \overline{B}$.
Hence $\overline{Z} = \overline{I_{D,A}} \cup \overline{I_{E,B}} \le_{w\alpha e} \overline{A} \oplus \overline{B}$ as required.

iii)
If $Z \in \Sigma_1(L_\alpha)$, then $Z \in \Sigma_1(L_\alpha)$ and $A \in \Sigma_1(L_\alpha)$, $B \in \Sigma_1(L_\alpha)$ which contradicts the assumption. Hence $Z \not \in \Sigma_1(L_\alpha)$.

iv) From ii) and megaregularity of $A \oplus B$, we have $\overline{Z} \le_{\alpha e} \overline{A} \oplus \overline{B}$. Note $\overline{A \oplus B} = \overline{A} \oplus \overline{B}$. Combining this with i) it yields $Z \le_\alpha A \oplus B$. Hence $Z \in \Delta_1(L_\alpha, A, B)$. If $Z$ was bounded, then by \cref{prop_alpha_fin_iff_bounded_and_alpha_comp_with_oracle} using the megaregularity of $A \oplus B$, $Z$ is $\alpha$-finite. This contradicts iii). Hence $Z$ has to be unbounded.
\end{proof}

\begin{defn}(Weak halting set)\\
The \emph{weak halting set} is defined as $K(A):=\{x < \alpha : x \in \Phi_x(A)\}$.
\end{defn}

\begin{thm}\label{thm_kalimullin_thm2.5}\footnote{Theorem 2.5 in \cite{sh2003definability} for $\alpha=\omega$.}
Let $A, B, U \subseteq \alpha$. Let one of the conditions hold:\\
i) the projectum of $\alpha$ is $\alpha^* = \omega$ and $U^+$ is megaregular.\\
ii) $A \oplus B \oplus K(U)$ is megaregular.

Suppose $\neg \ukpair{U}{A}{B}$.
Then $\exists X, Y \subseteq \alpha [Y \le_{\alpha e} X \oplus A \land Y \le_{\alpha e} X \oplus B \land Y \not\le_{w \alpha e} X \oplus U]$.
%Moreover, sets $X, Y$ may be taken to be computable in $A \oplus B \oplus K(U)$.
\end{thm}

The following proof is a generalization of the proof for the case when $\alpha=\omega$ in \cite{sh2003definability}.

\begin{proof}
We perform a construction in $\alpha^*$ stages and define sets $X, Y$ st $\forall x < \alpha$:\\
\begin{instat}\label{k_t2.5_s2.1}
x \in Y \iff x \in X \land D_x \subseteq A \iff x \in X \land E_x \subseteq B
\end{instat}\\
which guarantees $Y \le_{\alpha e} X \oplus A$ and $Y \le_{\alpha e} X \oplus B$ by \cref{lemma_y_ae_from_x_join_a}.

Index the requirements and $\alpha$-enumeration operators by indices in $\alpha^*$ using \cref{prop_indexing_ace_sets_with_proj}.
Aim to meet for all $e <\alpha^*$ the requirements
\begin{align*}R_e: Y \not=\aeop{e}(X \oplus U).\end{align*}
At each stage $s < \alpha^*$ of the construction aim to define an $\alpha$-finite set $X_s$ and an $\alpha$-computable set $M_s$ so that for all $s < \alpha^*$ they satisfy:

\begin{gather}
X_s \subseteq X_{s+1}\label{k_t2.5_s2.2}\\
M_{s+1} \subseteq M_s\label{k_t2.5_s2.3}\\
X_{s+1} - X_s \subseteq M_{s+1}\label{k_t2.5_s2.4}\\
\forall D, E \in L_\alpha [D \subseteq A \land E \subseteq B \implies \exists x \in M_s [D \subseteq D_x \subseteq A \land E \subseteq E_x \subseteq B]]\label{k_t2.5_s2.5}\\
X_s \in L_\alpha\label{k_t2.5_s_x_s}\\
N_s \in L_\alpha\label{k_t2.5_s_n_s}\\
I_s \in L_\alpha\label{k_t2.5_s_i_s}\\
M_s := (\bigcap_{x \in I_s} V_x) - N_s = V_z - N_s\label{k_t2.5_m_s_eq}\\
M_s \in \Delta_1(L_\alpha)\label{k_t2.5_s_m_s}
\end{gather}

\subsubsection{Pre-construction}

By \cref{k_t2.5_m_s_eq}, the set $M_s$ is defined at every stage $s < \alpha^*$ by the sets $N_s$ and $I_s$.
Since the set $I_s$ is $\alpha$-finite at the stage $s$ by \cref{k_t2.5_s_i_s}, so by \cref{lemma3_thm_intersection_of_vx_alpha_fin} there is an index $z$ which is $\alpha$-computable from $I_s$ and $V_z = \bigcap_{x \in I_s} V_x$. Hence the equality $(\bigcap_{x \in I_s} V_x) - N_s=V_z - N_s$ holds at every stage $s$ where $I_s \in L_\alpha$. Consequently also the set $V_z$ is $\alpha$-computable at such stage $s$.

Since the set $N_s$ is $\alpha$-finite by \cref{k_t2.5_s_n_s} and $V_z$ is $\alpha$-computable at the stage $s$, so the set $M_s$ has to be $\alpha$-computable at the stage $s$, hence \cref{k_t2.5_s_m_s} holds.

When proving at the stage $s < \alpha^*$ that \cref{k_t2.5_s2.5} holds, we use the fact that $A$ and $B$ are not $\alpha$-finite by \cref{prop_trivial_u_k_pair} since $\neg \ukpair{U}{A}{B}$. This given $\alpha$-finite sets $D, E$, enables us to find arbitrarily large $\alpha$-finite supersets of $D, E$ contained in $A$ and $B$ respectively.

\subsubsection{Constructing $X$}
The set $X$ will be constructed in $\alpha^*$-many stages.
\begin{itemize}
\item Stage $s=0$.
Set $X_0 := \emptyset$, $N_0 := \emptyset$, $I_0 := \emptyset$.
Observe \cref{k_t2.5_s2.5} is true for $M_0=\alpha$.
Clearly, \cref{k_t2.5_s_x_s,k_t2.5_s_n_s,k_t2.5_s_i_s} are satisfied.

\item Stage $s+1=3e > 0$, $3e$ is a successor ordinal.
Define $X_{s+1} := X_s$, $N_{s+1} := N_s$, $I_{s+1} := I_s$.
Since the sets $X_{s+1}$, $N_{s+1}$, $I_{s+1}$ are the same as the sets $X_{s}$, $N_{s}$, $I_s$ and \cref{k_t2.5_s2.2,k_t2.5_s2.3,k_t2.5_s2.4,k_t2.5_s2.5,k_t2.5_s_x_s,k_t2.5_s_n_s,k_t2.5_s_i_s} hold at the stage $s$ by IH, they hold at the stage $s+1$ too.

\item Stage $s+1=3e+1$.
By induction hypothesis let $X_s$, $N_s$, $I_s$ be given and $\alpha$-finite by \cref{k_t2.5_s_x_s,k_t2.5_s_n_s,k_t2.5_s_i_s}.
Define $X_{s+1}:=X_s$, $N_{s+1} := N_s \cup \{e\}$, $I_{s+1} := I_s$.
Trivially, \cref{k_t2.5_s_x_s,k_t2.5_s_n_s,k_t2.5_s_i_s} hold at the stage $s+1$ by IH at the stage $s$.

Note $M_{s+1} = M_s - \{e\}$ by \cref{k_t2.5_m_s_eq}.
We claim that the set $M_{s+1}$ satisfies \cref{k_t2.5_s2.5}.
Let $D, E \in L_\alpha \land D \subseteq A \land E \subseteq B$.
By IH on $M_s$ there is $x \in M_s$ st $[D \subseteq D_x \subseteq A \land E \subseteq E_x \subseteq B]$.
Note $D_x \in L_\alpha$, but by \cref{prop_trivial_u_k_pair} $A \not\in L_\alpha$, hence $D_x \subset A$.
Let $z \in A - D_x$.
Then $\hat{D}:=D_x \cup \{z\} \in L_\alpha$.
By IH on $M_s$ there is $y \in M_s$ st $\hat{D} \subseteq D_y \subseteq A \land E \subseteq E_y \subseteq B$.
If $x \not=e$, then $x \in M_{s+1} := M_s - \{e\}$.
Otherwise $x = e \not= y$ and $y \in M_{s+1} \land D \subseteq D_x \subset \hat{D} \subseteq D_y \subseteq A \land E \subseteq E_y \subseteq B$.
Therefore in any case the set $M_{s+1}$ satisfies \cref{k_t2.5_s2.5}.

\item Stage $s+1=3e+2$.
Aim to find $x\in M_s$ st one of the two following statements is true:\\
1: $D_x \not\subseteq A \land E_x \not\subseteq B \land x \in \aeop{e}((X_s \cup (M_s \cap V_x)) \oplus U)$,\\
2: $D_x \subseteq A \land E_x \subseteq B \land x \not\in \aeop{e}((X_s \cup (M_s \cap V_x)) \oplus U)$.\\
First we prove the existence of such $x \in M_s$.
Assume that $\forall x \in M_s$ the statement 2 is false.
Define
\begin{align*}
W := \{\langle a,b \rangle : \exists x \in M_s [a \in D_x \land b \in E_x \land x \in \aeop{e}((X_s \cup (M_s \cap V_x)) \oplus U)]\}.
\end{align*}

Then $W \le_{\alpha e} U$ by the regularity of $U^+$, \cref{lemma2_thm_kalimullin_thm2.5}, \cref{k_t2.5_s_m_s} and \cref{k_t2.5_s_x_s}.

We prove $A \times B \subseteq W$.
Let $(a,b) \in A \times B$.
By \cref{k_t2.5_s2.5} for $M_s$ it follows $\exists x \in M_s [a \in D_x \subseteq A \land b \in 	E_x \subseteq B]$.
Since statement 2 is false, we have $x \in \aeop{e}((X_s \cup (M_s \cap V_x)) \oplus U)$.
Thus $(a,b) \in W$.
Since $\neg \ukpair{U}{A}{B}$, there is a pair $(a,b) \in \overline A \times \overline B$ st $(a,b) \in W$.
Thus there is $x \in M_s$ st $a \in D_x, b\in E_x$ and $x \in \aeop{e}((X_s \cap (M_s \cap V_x)) \oplus U)$.
Hence $D_x \not\subseteq A$, $E_x \not\subseteq B$ and statement 1 is true for $x \in M_s$.
Therefore there is $x \in M_s$ st statement 1 or statement 2 is true.
Choose such an element $x \in M_s$ using the oracle $A \oplus B \oplus K(U)$.

Case 1: If statement 1 is true for $x$, then $x \in \aeop{e}((X_s \cup (M_s \cap V_x) \oplus U)$.
By \cref{prop_enum_op_witness} and \cref{prop_enum_op_monotonicity} there is $F \subseteq X_s \cup (M_s \cap V_x)$ st $F \in L_\alpha \land x \in \aeop{e}(F \oplus U)$.
Thus define $X_{s+1} := X_s \cup F$, $N_{s+1} := N_s$, $I_{s+1}:=I_s$. Note that$M_{s+1} := M_s$.
The set $F$ is $\alpha$-finite, by IH $X_s$ is $\alpha$-finite and so the union $X_{s+1} = X_s \cup F$ is $\alpha$-finite satisfying \cref{k_t2.5_s_x_s}.
\Cref{k_t2.5_s_n_s,k_t2.5_s_i_s} are true by IH.

Case 2: Otherwise if statement 2 is true for $x$, then define $X_{s+1} := X_s \cup \{x\}$, $N_{s+1} := N_s$, $I_{s+1} := I_s \cup \{x\}$.
Trivially, the sets $X_{s+1}, N_{s+1}, I_{s+1}$ are $\alpha$-finite using IH, hence satisfying \cref{k_t2.5_s_x_s,k_t2.5_s_n_s,k_t2.5_s_i_s}.
Note $M_{s+1} = M_s \cap V_x$ by \cref{k_t2.5_m_s_eq}.
$M_{s+1}$ satisfies \cref{k_t2.5_s2.5}: if $D \subseteq A, E\subseteq B, D \in L_\alpha, E \in L_\alpha$, then by the hypothesis on $M_s$, there is $y \in M_s$ st $D \cup D_x \subseteq D_y \subseteq A$ and $E \cup E_x \subseteq E_y \subseteq B$.
Therefore $y \in M_s \cap V_x = M_{s+1}$.

Note in both cases $X_{s+1}-X_s \subseteq M_{s+1}$ \cref{k_t2.5_s2.4} being satisfied.

%TODO prove that 2.5 holds in the limit case so that the requirement $\alpha^*=\omega$ could be removed.
\item Stage $s=3e > 0$, $3e$ is a limit ordinal.
If $\alpha^*=\omega$, then this stage does not arise. Hence assume that $A \oplus B \oplus K(U)$ is megaregular.

Define $X_s := \bigcup_{r < s}X_r,  N_s := \bigcup_{r < s}N_r, I_s := \bigcup_{r < s}I_r$.
We claim that these sets are $\alpha$-finite.

Define a partial function $f:\alpha \rightharpoonup \alpha$ on the ordinals smaller than $s$ by $f(r)=\{\gamma < \alpha : K_\gamma = X_r\}$.
Note that by IH for all $r < s$, the set $X_r$ is $\alpha$-finite using \cref{k_t2.5_s_x_s}.
Also during the construction we only use the oracle $A \oplus B \oplus K(U)$. Thus the index $f(r)$ of an $\alpha$-finite set $X_r$ is also $A \oplus B \oplus K(U)$-computable. Consequently, the function $f$ is $\Sigma_1(L_\alpha, A \oplus B \oplus K(U))$ definable. As $s < \alpha^*$, so $s$ as a limit ordinal is an $\alpha$-finite set. Therefore by the megaregularity of $A \oplus B \oplus K(U)$, the set $f[s]$ is also $\alpha$-finite. But then $X_s = \bigcup_{\gamma \in f[s]} K_\gamma$ is $\alpha$-finite by \cref{prop_alpha_finite_union}. So \cref{k_t2.5_s_x_s} holds at the stage $s$ as required.
Applying similar reasoning, using the veracity of \cref{k_t2.5_s_n_s,k_t2.5_s_i_s} for all $r < s$ by IH, we conclude that \cref{k_t2.5_s_n_s,k_t2.5_s_i_s} hold at the stage $s$ too.

Note $M_s := \bigcap_{r < s}M_r$ by \cref{k_t2.5_m_s_eq}. We prove that \cref{k_t2.5_s2.5} holds at the stage $s$.
Note that $M_s = V_z - N_s$ by \cref{k_t2.5_m_s_eq} for some $z < \alpha$ satisfying both $D_z \subseteq A$ and $E_z \subseteq B$.
Fix $\alpha$-finite sets $D$ and $E$ st $D \subseteq A$ and $E \subseteq B$.
WLOG let $D_z \subseteq D$ and $E_z \subseteq E$.
Define $Z:=\{ x < \alpha : D \subseteq D_x \subseteq A \land E \subseteq E_x \subseteq B \}$.
As $\neg \kpair{A}{B}$ by the assumption, so $A \not\in \Sigma_1(L_\alpha)$ and $B \not\in \Sigma_1(L_\alpha)$ by \cref{prop_trivial_u_k_pair}.
Note that $A \oplus B$ is megaregular.
Hence $Z$ is unbounded by \cref{lemma4_thm_z_unbounded}.
On the other hand $N_s \subseteq s$. Thus $Z - N_s \not= \emptyset$.
Note $\{x \in M_s :D \subseteq D_x \subseteq A \land E \subseteq E_x \subseteq B\}=\{x \in V_z - N_s :D \subseteq D_x \subseteq A \land E \subseteq E_x \subseteq B\} = Z - N_s \not= \emptyset$.
Therefore $\forall D, E \in L_\alpha [D \subseteq A \land E \subseteq B \implies \exists x \in M_s [D \subseteq D_x \subseteq A \land E \subseteq E_x \subseteq B]]$ and so the statement \cref{k_t2.5_s2.5} is satisfied at the limit stage $s$.

\end{itemize}

Finally, define $X := \bigcup_{s < \alpha^*} X_s$.

\subsubsection{Defining $Y$}
To define $Y$ first prove $\forall z \in X [D_z \subseteq A \iff E_z \subseteq B]$:
Let $z \in X$.
Then there is a stage $s+1=3e+2$ st $z \in X_{s+1} - X_s$.
In case 2 $D_z \subseteq A$ and $E_z \subseteq B$.
In case 1 there is $x$ st $X_{s+1} - X_s \subseteq V_x, D_x \not\subseteq A$ and $E_x \not\subseteq B$.
As $z \in X_{s+1} - X_s \subseteq V_x$, so $D_x \subseteq D_z$ and $E_x \subseteq E_z$.
Thus $D_z \not\subseteq A$ and $E_z \not\subseteq B$.
Define the set
\begin{center}
$Y := \{z \in X : D_z \subseteq A\}=\{z \in X : E_z \subseteq B\}$.
\end{center}

\subsubsection{Final verification}
Note $Y \le_{\alpha e} X \oplus A$ and $Y \le_{\alpha e} X \oplus B$ as proved under \cref{k_t2.5_s2.1}.

We prove $Y \not\le_{w\alpha e} X \oplus U$ by showing $Y \not=\aeop{e}(X \oplus U)$ for an arbitrary $e < \alpha^*$.
Consider a stage $s+1=3e+2$.
In case 1 $X_{s+1}=X_s \cup F$ and there is $x$ st $x \in \aeop{e}(F \oplus U), D_x \not\subseteq A$ and $E_x \not\subseteq B$.
Hence $x \in \aeop{e}(X \oplus U) - Y$.
In case 2 there is $x$ st $X_{s+1}=X_s \cup \{x\}, M_{s+1}=M_s \cap V_x, D_x \subseteq A, E_x \subseteq B$ and $x \not\in \aeop{e}((X_s \cup M_{s+1}) \oplus U)$.

Let $z \in X$.
Then $\exists t. z \in X_{t+1} - X_t \subseteq M_{t+1}$ by \cref{k_t2.5_s2.4}.
If $t \ge s$, then $z \in M_{s+1}$ by \cref{k_t2.5_s2.3}.
If $t < s$, then $z \in X_s$ by \cref{k_t2.5_s2.2}.
Hence $z \in X_s \cup M_{s+1}$ and thus $X \subseteq X_s \cup M_{s+1}$.

Hence $x \in Y - \aeop{e}(X \oplus U)$ by \cref{prop_enum_op_monotonicity}.
Therefore in both cases $Y \not= \aeop{e}(X \oplus U)$ and so $Y \not\le_{w\alpha e} X \oplus U$.
\end{proof}

%done3, 27.3.2015
%TODO
%Needs to be verified that the enumerations can by done not only element-wise as int the proofs underneath, but with the alpha-finite subsets.
\begin{thm}\label{thm_k_pair_correspondence_ae}\footnote{From theorem 2.6 in \cite{sh2003definability} for $\mathcal{D}_T$.}
The statements i) - iv) are equivalent and imply v).
Moreover if the projectum of $\alpha$ is $\alpha^*=\omega$ and $U^+$ is megaregular or $A \oplus B \oplus K(U)$ is megaregular, then all the statements i) - v) are equivalent.

i) $\ukpair{U}{A}{B}$, i.e. $\exists W \le_{\alpha e} U. A \times B \subseteq W \land \overline{A} \times \overline{B} \subseteq \overline{W}$,

ii) $\exists f(x,y) \in \Delta_1(L_\alpha). \forall X \subseteq \alpha.$
$\forall x, y \in \alpha. \aeop{x}(A \oplus X) \cap \aeop{y}(B \oplus X)$ 
$\subseteq \aeop{f(x,y)}(X \oplus U)$
$\subseteq \aeop{x}(A \oplus X) \cup \aeop{y}(B \oplus X)$,

iii) $\exists f(x,y) \in \Delta_1(L_\alpha) \forall x, y < \alpha [\aeop{x}(A)=\aeop{y}(B) \implies \aeop{f(x,y)}(U) = \aeop{x}(A)]$,

iv) $\forall V_1, V_2 [V_1 \le_{\alpha e} A \land V_2 \le_{\alpha e} B \implies \exists W \le_{\alpha e} U. V_1 \cap V_2 \subseteq W \subseteq V_1 \cup V_2]$,

v) $\forall X \subseteq \alpha. \aedeg{X \oplus U} = \aedeg{A \oplus X \oplus U} \wedge \aedeg{B \oplus X \oplus U}$.
\end{thm}

\begin{proof}
The implications ii) $\implies$ iii), ii) $\implies$ iv), ii) $\implies$ v), iv) $\implies$ i) are trivial.
It remains to prove the implications i) $\implies$ ii) and iii) $\implies$ i).

i) $\implies$ ii):

Assume $\exists W \le_{\alpha e} U. A \times B \subseteq W \land \overline{A} \times \overline{B} \subseteq \overline{W}$ and let $W = \aeop{}(U)$ for some $\alpha$-enumeration operator $\aeop{}$.

Define $f$ st for any $X \subseteq \alpha$, $x,y \in \alpha$:\\
$\aeop{f(x,y)}(X \oplus V) := \{z \in \alpha : \exists D, E \in L_\alpha [z \in \aeop{x}(D \oplus X) \cap \aeop{y}(E \oplus X) \land D \times E \subseteq \aeop{}(V)\}$.

Then $f$ is $\alpha$-computable and satisfies the condition ii).

iii) $\implies$ i):
Suppose that $A$ and $B$ satisfy the condition iii) with $f$ being computable.
Define a computable function $g$ st for every $Y \subseteq \alpha$ and $y < \alpha$:

$\aeop{g(y)}(Y)=
\begin{cases}
\alpha & $if $y \in Y,\\
\emptyset & $otherwise.$
\end{cases}
$

Then $A, B$ are a $U$-Kalimullin pair with a witness $W=\{(m,n) : \aeop{f(g(m),g(n))}(U) \not=\emptyset\}$.

Hence i) $\iff$ ii) $\iff$ iii) $\iff$ iv) $\implies$ v) for any $\alpha$. 
Note v) $\implies$ i) is the contrapositive of \cref{thm_kalimullin_thm2.5}.
Therefore i) $\iff$ ii) $\iff$ iii) $\iff$ iv) $\iff$ v) if $\alpha^*=\omega$ and $U^+$ is megaregular or $A \oplus B \oplus K(\emptyset)$ is megaregular.
\end{proof}

The statement i) iff v) establishes the definability of a $U$-Kalimullin pair.

%done1, 27.3.2015
\begin{prop}\label{prop_kalimullin_pair_closure}\footnote{Proposition 1.7 in \cite{caidefining} for $\alpha=\omega$.}
Let $B \subseteq \alpha$.
The set of all $A$ st $\kpair{A}{B}$ is closed downwards under $\alpha$-enumeration reducibility as well as closed under join.
\end{prop}

\begin{proof}
Suppose $\kpair{A_0}{B}$ and $A_1 \le_{\alpha e} A_0$.
Hence $\exists W_0 \in \Sigma_1(L_\alpha). A_0 \times B \subseteq W_0 \land \overline A_0 \times \overline B \subseteq \overline W_0$. Let $V_1 := A_1 \times \alpha$, $V_2 := \alpha \times B$.
As $A_1 \le_{\alpha e} A_0$, so $V_1 \le_{\alpha e} A_0 \land V_2 \le_{\alpha e} B$. Hence by theorem \ref{thm_k_pair_correspondence_ae} (i implies iv), $\exists W_1 \in \Sigma_1(L_\alpha)$ st $V_1 \cap V_2 \subseteq W_1 \subseteq V_1 \cup V_2$.
Therefore $V_1 \cap V_2 = A_1 \times B \subseteq W_1$.
Also $W_1 \subseteq V_1 \cup V_2 \iff \overline V_1 \cap \overline V_2 \subseteq \overline W_1$ and so
$\overline V_1 \cap \overline V_2 = (\overline A_1 \times \alpha) \cap (\alpha \times \overline B) = \overline A_1 \times \overline B \subseteq \overline W_1$.
Hence $\kpair{A_1}{B}$.

Let $\kpair{A_0}{B} \land \kpair{A_1}{B}$. If $A_i = \emptyset$ for $i \in 2$ then $A_0 \oplus A_1 \equiv_{\alpha e} A_{1-i}$ and so $\kpair{A_0 \oplus A_1}{B}$. Otherwise $\kpair{A_0 \oplus A_1}{B}$ by lemma \ref{lemma_k_pair_distributivity}.
\end{proof}

%done1, 27.3.2015
\begin{cor}\label{cor_k_pair_definability}(Definability of a Kalimullin Pair)\footnote{The case for $\alpha=\omega$ proved in \cite{sh2003definability}.}\\
Let $\alpha^* = \omega$ or assume $V=L$ and let $\alpha$ be an infinite regular cardinal. Then:\\
$\forall a, b \in \mathcal{D}_{\alpha e} [\kpair{a}{b} \iff
\forall x \in \mathcal{D}_{\alpha e}. (a \vee x) \wedge (b \vee x) = x]$.
\end{cor}

\begin{proof}
Note that if $U \in \Sigma_1(L_\alpha)$, then $U^+$ is megaregular by \cref{cor_mr_deg_invariance}.
Thus the statement above follows from $(i \iff v)$ in \cref{thm_k_pair_correspondence_ae} and from the $\mathcal{K}$-pair being a degree theoretic property by its invariance under the $\alpha e$-reducibility by proposition \ref{prop_kalimullin_pair_closure}.
\end{proof}

\begin{cor}\label{cor_uk_pair_definability}(Definability of an $U$-Kalimullin Pair)\\
Assume $V=L$ and let $\alpha$ be an infinite regular cardinal. Then:\\
$\forall a, b, u \in \mathcal{D}_{\alpha e} [\ukpair{u}{a}{b} \iff
\forall x \in \mathcal{D}_{\alpha e}. (a \vee x \vee u) \wedge (b \vee x \vee u) = x \vee u]$.
\end{cor}

\begin{proof}
Note that since $\alpha$ is an infinite regular cardinal, so $A \oplus B \oplus K(U)$ is megaregular.
Thus the statement above follows from $(i \iff v)$ in \cref{thm_k_pair_correspondence_ae} and from the $\mathcal{K}$-pair being a degree theoretic property by its invariance under the $\alpha e$-reducibility by proposition \ref{prop_kalimullin_pair_closure}.
\end{proof}

%%%%%%%%%%%%%%%%%%%%%%%%%%%%%%%
\subsection{Maximal Kalimullin pair and total degrees}%%%%%%
%%%%%%%%%%%%%%%%%%%%%%%%%%%%%%%

%done1, 26.3.2015
%relies on unverified prop_kalimullin_pair_closure, prop_ch2_cai1.8.
\begin{prop}\label{lemma_sc_max_k_pair}(Maximality of semicomputable megaregular Kalimullin pairs)\footnote{Generalized from Maximal $\mathcal{K}$-pairs in \cite{caidefining} for $\alpha=\omega$.}\\
Let $A \subseteq \alpha$ and let $A^+$ and $A^-$ be both megaregular.
If $\kpair{A}{\overline A} \land A \not\in \Sigma_1(L_\alpha) \land A \not\in \Pi_1(L_\alpha)$, then $\kpairmax{A}{\overline{A}}$.
\end{prop}

\begin{proof}
Suppose $\kpair{A}{\overline A}$ and $\kpair{C}{D}$, $A \le_{\alpha e} C$, $\overline A \le_{\alpha e} D$. By \cref{prop_kalimullin_pair_closure} $\kpair{A}{D}$. By \cref{prop_ch2_cai1.8_cor} and megaregularity of $A^-$ we have $D \le_{\alpha e} \overline A$. Similarly, $\kpair{\overline A}{C}$ and thus $C \le_{\alpha e} \overline{\overline A} = A$ by the megaregularity of $A^+$.
\end{proof}

%done2
\begin{cor}\label{thm_total_deg_join_of_max_k}
Let $\alpha^* = \omega$ or assume $V=L$ and let $\alpha$ be an infinite regular cardinal.
Then every nontrivial total megaregular degree is a join of a maximal $\mathcal{K}$-pair, i.e.\\
$\forall a \in \mrtotalaedegs - \{0\} \exists b, c \in \mathcal{D}_{\alpha e} [(a = b \vee c) \land \kpairmax{b}{c}]$.
\end{cor}
\begin{proof}
Since $\alpha^* = \omega$ or $\alpha$ is an infinite regular cardinal, thus a (maximal) Kalimullin pair is definable by \cref{cor_k_pair_definability}.

Suppose $a \in \mathcal{TOT}_{\alpha e} - \{0\}$ and $a$ is a megaregular degree (at least one or equivalently every set in $a$ is megaregular).
Then by theorem \ref{thm_jockusch_strongly_non-re_semicomputable}, there is $A \subseteq \alpha$ st $A$ $\alpha$-semicomputable, $A \not\in \Sigma_1(L_\alpha)$, $\overline A \not\in \Sigma_1(L_\alpha)$ and $A \oplus \overline{A} \in a$ by the totality of $a$.
As $A$ $\alpha$-semicomputable, so $\kpair{A}{\overline A}$ by proposition \ref{prop_sc_pair_kalimullin}.
$\kpair{A}{\overline{A}}$ is nontrivial since $A \not\in \Sigma_1(L_\alpha)$ and $\overline{A} \not\in \Sigma_1(L_\alpha)$.
Thus by \cref{lemma_sc_max_k_pair} and the megaregularity of $A$ we have $\kpairmax{A}{\overline A}$.
\end{proof}

By inspecting whether a degree which is not quasiregular could be a join of a maximal Kalimullin pair, one may establish the following:

\begin{prop}
If $\adeg{B}$ is not a quasiregular degree, then there is $C$ st $0 <_\alpha C <_\alpha B$ and $\kpairmax{C}{\overline{C}}$.
\end{prop}

\begin{proof}
Since $\adeg{B}$ is not a quasiregular degree, then $D$ is not quasiregular for any $D \equiv_\alpha B$.
So $B$ is not quasiregular.

Let $\beta < \alpha$ be the least ordinal st $B \cap \beta \not\in L_\alpha$.
Define $A := B \cap \beta$.
Then $A \subset B$ by $B$ not being quasiregular.
By the minimality of $\beta$, the set $A$ is quasiregular.
$A$ is bounded, but not $\alpha$-finite, hence $A$ cannot be $\alpha$-computable.
Thus $A >_\alpha \emptyset$.
By \cref{thm_jockusch_strongly_non-re_semicomputable} there is $\alpha$-semicomputable set $C$ st $A \equiv_\alpha C$, $C \not\in \Sigma_1(L_\alpha)$ and $C \not\in \Pi_1(L_\alpha)$.
As $C$ is $\alpha$-semicomputable, so $\kpair{C}{\overline{C}}$.
By \cref{lemma_sc_max_k_pair} we have that $\kpairmax{C}{\overline{C}}$.
\end{proof}

\section{Acknowledgements}
The author would like to thank Mariya Soskova and Hristo Ganchev for the explanation of the proof of the Kalimullin pair definability in the classical case, i.e. $\alpha = \omega$.

The author was supported by Hausdorff Research Institute for Mathematics during Hausdorff Trimester Program \emph{Types, Sets and Constructions}.

\bibliographystyle{plain}
\bibliography{References/references} % References file
\addcontentsline{toc}{chapter}{References} % Adds References to contents page

\end{document}